\documentclass[12pt,a4paper]{amsart}
\usepackage{graphics}
\usepackage{epsfig}
\usepackage[all]{xy}
\usepackage{graphicx}
\theoremstyle{plain}
\usepackage{amssymb}
\usepackage[english]{babel}
\advance\hoffset-20mm \advance\textwidth40mm

\newtheorem{theorem}{Theorem}
\newtheorem{lemma}{Lemma}

\begin{document}
	\sloppy
	\title[Isomorphisms of groups of periodic infinite matrices]
	{Isomorphisms of groups of periodic infinite matrices}
	\author
	{Oksana Bezushchak}
	\address{Oksana Bezushchak: Faculty of Mechanics and Mathematics, Taras Shevchenko National University of Kyiv, Volodymyrska, 60, Kyiv 01033, Ukraine}
	\email{bezushchak@knu.ua}

	\keywords{infinite matrix; isomorphism; locally matrix algebra; Steinitz number}
	\subjclass[2020]{20E34, 20H20}

	\maketitle
	
	\begin{abstract}
	We describe isomorphisms of groups of several periodic  infinite matrices and isomorphisms of groups of invertible elements of unital locally matrix algebras.		
	\end{abstract}

	\section*{Introduction}

Let $\mathbb{N}$ be the set of positive integers. A   \emph{Steinitz number} \cite{ST}   is an infinite formal
product of the form
$$ \prod_{p\in \mathbb{P}} p^{r_p}, $$
where $ \mathbb{P} $ is the set of all primes, $ r_p \in  \mathbb{N} \cup \{0,\infty\} $ for all $p\in \mathbb{P}$.
We can define the product of two Steinitz numbers by the rule:
$$ \prod_{p\in \mathbb{P}} p^{r_p} \cdot  \prod_{p\in \mathbb{P}} p^{k_p}= \prod_{p\in \mathbb{P}} p^{r_p+k_p}, \quad  r_p, k_p \in  \mathbb{N} \cup \{0,\infty\},  $$
where we assume that
$$r_p+k_p=\begin{cases}
	r_p+k_p & \text{if  $r_p < \infty$ and $k_p < \infty$, } \\
	\infty & \text{in other cases.}
\end{cases}$$
The Steinitz number $s_2$ \emph{divides} the Steinitz number $s_1$ (denote as $s_2|s_1$) if there exists the Steinitz number $ s_3\in\mathbb{SN} $ such that $s_1 = s_2 \cdot s_3.$

Let $\mathbb{F} $ be a field. In what follows we always assume that $\text{char}\, \mathbb{F} \not = 2,3.$ We call an infinite $(\mathbb{N} \times \mathbb{N})$-matrix $A$ over the field $\mathbb{F} $ \emph{periodic} if it is block-diagonal $A=\text{diag}(a,a,\ldots),$ where $a$ is an $( n \times n)$-matrix for some $n\in \mathbb{N}$. The number $n$ is called a \emph{period} of the matrix $A$ and the matrix $A$ is called $n$-\emph{periodic}.

Let $M_n ^p (\mathbb{F} )$ be the algebra of all $n$-periodic $(\mathbb{N} \times \mathbb{N})$-matrices, and let $M_n  (\mathbb{F} )$ be the algebra of all $(n\times n)$-matrices over a field $\mathbb{F} $. Clearly, $$M_n ^p (\mathbb{F} ) \cong M_n(\mathbb{F}), \quad \text{and} \quad M_n ^p (\mathbb{F} )\subseteq M_m ^p (\mathbb{F} ) \quad \text{if and only if} \quad n \quad \text{divides} \quad  m.$$ Let $GL_n^p(\mathbb{F})$ be the group of all invertible matrices in  $M_n ^p (\mathbb{F} ),$ and let $GL_n(\mathbb{F})$ be the group of all invertible matrices in  $M_n  (\mathbb{F} ).$  It is easy to see that $GL_n^p(\mathbb{F}) \cong GL_n(\mathbb{F}).$

For a Steinitz number  $s$, we consider the algebra $$ \quad M_s ^p (\mathbb{F} ) =\bigcup_{n|s} M_n ^p (\mathbb{F} ), $$ and the group $$ GL_s^p(\mathbb{F})=\bigcup_{n|s} GL_n^p(\mathbb{F})$$ that consists of all invertible elements of the algebra $M_s ^p (\mathbb{F} ).$ Let $$SL_n(\mathbb{F})=[GL_n(\mathbb{F}),GL_n(\mathbb{F})],$$ $$SL_n^p(\mathbb{F})= [\,GL_n^p(\mathbb{F}),GL_n^p(\mathbb{F})\,] \quad \text{and} \quad  SL_s^p(\mathbb{F})= [\,GL_s^p(\mathbb{F}),GL_s^p(\mathbb{F})\,] $$ be  commutator subgroups of  groups $GL_n(\mathbb{F}),$ $GL_n^p(\mathbb{F})$ and  $GL_s^p(\mathbb{F}),$  respectively.

If $n_1 < n_2 < \cdots $ is a sequence of positive integers such that $n_i | n_{i+1},$ $i\geq 1,$ and $s$ is the least common multiple of $n_1,$ $n_2,$ $\ldots,$ then $$ GL_{n_1}^p(\mathbb{F}) \subset GL_{n_2}^p(\mathbb{F})\subset \cdots , \quad \bigcup_{i \geq 1 } GL_{n_i}^p(\mathbb{F})= GL_{s}^p(\mathbb{F});$$  $$ SL_{n_1}^p(\mathbb{F}) \subset SL_{n_2}^p(\mathbb{F})\subset \cdots , \quad \bigcup_{i \geq 1 } SL_{n_i}^p(\mathbb{F})= SL_{s}^p(\mathbb{F}).$$ For more information about groups $GL_{s}^p(\mathbb{F}),$ $SL_{s}^p(\mathbb{F}),$ see \cite{Bez_Phys_Math,Bez_Sushch}.

Recall some definitions and facts concerning locally matrix algebras; see \cite{BezReportsNAS,14,BezOl_2,BezOl}.

An associative $\mathbb{F}$-algebra $A$ with the unit 1 is said to be a \textit{unital locally matrix algebra} (see \cite{Kurosh}) if for an arbitrary finite collection of elements $a_1,\ldots,a_t \in A$ there exists a subalgebra $A' \subset A$ such that $1, a_1,\ldots, a_t \in A'$ and $A'\cong M_n(\mathbb{F})$ for some $n \in \mathbb{N}.$ 

For a unital locally matrix algebra $A$, let $D(A)$ be the set of all positive integers $n$ such that there exists a subalgebra $A',$  $1 \in A' \subset A,$ $A' \cong M_n(\mathbb{F}).$ The least common multiple of the set $ D(A)$ is called the \textit{Steinitz number} $\mathbf{st}(A)$ \textit{of the algebra} $A$; see \cite{BezOl}.

J.~G.~Glimm \cite{Glimm} proved that every  countable-dimensional unital locally matrix algebra is uniquely determined by its Steinitz number. Remark, that the algebra $M_s^p(\mathbb{F})$ is a  countable-dimensional unital locally matrix algebra and  $\mathbf{st}(M_s^p(\mathbb{F}))=s.$ 

Let $A$ be a unital locally matrix algebra over a field $\mathbb{F}$. Let us denote by the symbol $A^{*}$ the group  of invertible elements of the algebra $A$ and let $[A^{*},A^{*}]$ be its commutator subgroup. Our aim now is description of isomorphisms of the group $A^{*}.$

Let $R,$ $S$ be rings. A mapping $\varphi:R \rightarrow S$ is called an \textit{anti-isomorphism} if \begin{enumerate}
	\item[(1)] $\varphi$ is an isomorphism of additive groups of $R$ and $S,$
	\item[(2)] $\varphi(ab)=\varphi(b)\varphi(a)$ for arbitrary elements $a,b\in R.$ 
\end{enumerate}

\begin{theorem}\label{Th_1_per_matrix} Let $A$ and $B$ be unital locally matrix $\mathbb{F}$-algebras. If  groups $[A^{*},A^{*}]$ and $[B^{*},B^{*}]$ are isomorphic, then  rings $A$ and $B$ are isomorphic or anti-isomorphic. Moreover, for an arbitrary isomorphism $\varphi:[A^{*},A^{*}] \rightarrow [B^{*},B^{*}]$ either there exists an isomorphism  of rings $\theta_1 :A \rightarrow B$ such that $\varphi$ is the restriction of  $\theta_1$ to $[A^{*},A^{*}]$ or there exists an anti-isomorphism of rings $\theta_2 :A \rightarrow B$ such that for an arbitrary element $g\in [A^{*},A^{*}]$ we have $$ \varphi(g) = \theta_2(g^{-1}).$$
\end{theorem}	
 
If  algebras $A$ and $B$ are countable-dimensional, then Theorem \ref{Th_1_per_matrix} can be made more precise. In this case, without loss of generality, we can assume that $A=M_s^p(\mathbb{F}),$ where $s$ is the Steinitz number of the algebra $A.$ The algebra $M_s^p(\mathbb{F})$ is invariant with respect to transpose $t$, which is an anti-isomorphism.

\begin{theorem}\label{Th_2_per_matrix} Let $A$ and $B$ be  countable-dimensional unital locally matrix $\mathbb{F}$-algebras. If  groups $[A^{*},A^{*}]$ and $[B^{*},B^{*}]$ are isomorphic, then  rings $A$ and $B$ are isomorphic. Moreover,  an arbitrary isomorphism $\varphi:[A^{*},A^{*}] \rightarrow [B^{*},B^{*}]$ either extends to  an isomorphism of rings $A \rightarrow B$ or there exists an isomorphism of rings  $\theta :A \rightarrow B$ such that for an arbitrary element $g\in [A^{*},A^{*}]$ we have $$ \varphi(g) = \theta\big( (g^{-1})^t\big).$$
\end{theorem}

If countable-dimensional unital locally matrix algebras are isomorphic as rings then they are isomorphic as $\mathbb{F}$-algebras; see Lemma \ref{lemma2_per_matrix} below. Therefore,  Theorem \ref{Th_2_per_matrix} implies  

\begin{theorem}\label{Th_3_per_matrix} Groups $SL_{s_1}^p(\mathbb{F})$ and $SL_{s_2}^p(\mathbb{F})$ are isomorphic if and only if $s_1 =s_2.$
\end{theorem}

For description of isomorphisms between groups $GL_s^p(\mathbb{F}),$ we need to introduce the concept of a central homothety.

For a unital $\mathbb{F}$-algebra $A,$ by a \emph{central homothety} of its multiplicative group $A^{*}$ we mean a multiplicative homomorphism $$ A^{*} \big/ [A^{*},A^{*}] \rightarrow \mathbb{F}^{*}.$$

\begin{theorem}\label{Th_4_per_matrix} Let $A$ and $B$ be  unital locally matrix $\mathbb{F}$-algebras. For an arbitrary isomorphism of multiplicative groups $\varphi: A^{*} \rightarrow B^{*}$ there  exists a central homothety $\chi: A^{*} \big/ [A^{*},A^{*}] \rightarrow \mathbb{F}^{*}$ and an isomorphism of rings $\theta_1: A \rightarrow B$ such that  $$ \varphi(g) = \chi(g)\theta_1(g), \quad g\in A^{*},$$ or an anti-isomorphism  $\theta_2 :A \rightarrow B$ such that   $$ \varphi(g) =\chi(g) \theta_2(g^{-1}), \quad g\in A^{*}.$$
\end{theorem}

As above, in the countable-dimensional case we can be more precise.

\begin{theorem}\label{Th_5_per_matrix} Let $A$ and $B$ be countable-dimensional unital locally matrix $\mathbb{F}$-algebras, and let $s$ be the Steinitz number of $A$. For an arbitrary isomorphism of multiplicative groups $\varphi: A^{*} \rightarrow B^{*}$ there  exists a central homothety $\chi:A^{*} \big/ [A^{*},A^{*}] \rightarrow \mathbb{F}^{*}$ and an isomorphism of rings $\theta: A \rightarrow B$ such that $$ \varphi(g) = \chi(g)\theta(g) \quad \text{for all} \quad  g\in A^{*}, \quad \text{or}$$    $$ \varphi(g) =\chi(g) \theta\big( (g^{-1})^t\big)  \quad \text{for all} \quad  g\in A^{*}.$$
\end{theorem}

Our final gool is description of the group of automorphisms of $SL_s^p(\mathbb{F}).$  The automorphisms group $\text{Aut}_{\mathbb{F}}\big(M_s^p(\mathbb{F})\big)$ of the $\mathbb{F}$-algebra  $M_s^p(\mathbb{F})$ has been described in \cite{14}.

\begin{theorem}\label{Th_6_per_matrix} Let $H$ be the cyclic group of order $2$ generated by the  automorphism $\psi:g \mapsto (g^{-1})^t,$ $g\in SL_s^p(\mathbb{F}).$ Then $$\emph{\text{Aut}}\big(SL_s^p(\mathbb{F})\big)=H\cdot \emph{\text{Aut}}_{\mathbb{F}}\big(M_s^p(\mathbb{F})\big)\cdot \emph{\text{Aut}}(\mathbb{F}).$$
\end{theorem}

\section{Isomorphisms of invertible elements groups of unital locally matrix algebras}

Description of isomorphisms of groups $SL_n(\mathbb{F})$ and $GL_n(\mathbb{F})$ is well known; see \cite{Borel}. It is also  easy to see that a group  $SL_n(\mathbb{F})$  is not isomorphic to the union of an infinite ascending chain  $$SL_{n_1}(\mathbb{F})  \subset SL_{n_2}(\mathbb{F})\subset \cdots, \quad n_1 < n_2 < \cdots . $$  Therefore, in proofs of Theorems \ref{Th_1_per_matrix},  \ref{Th_2_per_matrix} and Lemma \ref{lemma1_per_matrix} (see below) we will assume that the algebras $A$ and $B$ are infinite-dimensional. Hence, there exists a matrix subalgebra $1\in M_n(\mathbb{F})\subset A,$ $n \geq 4.$ Let $A'$ be the centralizer of the subalgebra $M_n(\mathbb{F})$ in $A$. By Joseph H.~M.~Wedderburn's Theorem (see \cite{Drozd_Kirichenko,Gol_Mikh_1,Pierce}), the algebra $A$ is isomorphic to $$  M_n(\mathbb{F}) \otimes_{\mathbb{F}} A' \cong M_n(A'). $$ We will identify the algebra $A$ with $M_n(A').$

Recall that for an arbitrary associative ring $R$ with $1$ and an arbitrary positive integer $k \geq 2$ the \emph{elementary linear group} $E_k (R)$ is the group generated by all transvections $$t_{ij} (a)=I_k + e_{ij} (a), $$ where $I_k$ is the identity $(k\times k)$-matrix, $1 \leq i\neq j \leq k,$ $a\in R,$ and $e_{ij}(a)$ is the  $(k\times k)$-matrix that has the element $a$ at the intersection of the $i$-th row and $j$-th column and zeros everywhere else. 
 
\begin{lemma}\label{lemma1_per_matrix}	$[A^{*},A^{*}]=E_n(A').$\end{lemma}
 
\begin{proof} Consider an arbitrary transvection $t_{ij}(a),$ $1 \leq i\neq j \leq n,$ $a\in A'.$ There exists a positive integer $r,$ $1 \leq r\leq n ,$ that is distinct from $i$ and $j.$ Then $t_{ij}(a)=[t_{ir}(1),t_{rj}(a)].$ We showed that $E_n(A')\subseteq [A^{*},A^{*}].$
	
Now, consider an arbitrary element $g\in [A^{*},A^{*}].$  There exists a subalgebra $M_n(\mathbb{F})\subset M_q(\mathbb{F})\subset  A$ such that $g\in SL_q(\mathbb{F})=E_q (\mathbb{F}).$ Consider a transvection $t_{ij}(\alpha)$	of the algebra $M_q(\mathbb{F}),$ $1 \leq i\neq j \leq q,$ $\alpha \in \mathbb{F}.$

The algebra $M_n(\mathbb{F})$ is embedded in the algebra  $M_q(\mathbb{F})$ diagonally, $$ a\rightarrow \text{diag}(\underbrace{a,a,\ldots, a}_k), \quad k=\frac{q}{n}, \quad a\in M_n(\mathbb{F}).$$ Hence, the matrix unit $e_{ii}(1)$ of the algebra $M_n(\mathbb{F}),$ $1 \leq i \leq n,$ is mapped into the element $$\overline{e}_i= \text{diag}(\underbrace{e_{ii}(1),\ldots, e_{ii}(1)}_k) \in  M_q(\mathbb{F}).$$  We have $\overline{e}_i M_q(\mathbb{F}) \overline{e}_i \cong M_k(\mathbb{F})$ and the algebra $M_q(\mathbb{F})$ can be identified with the algebra $M_n\big(M_k(\mathbb{F})\big),$ where $M_k(\mathbb{F})\cong A' \cap M_q(\mathbb{F})$ is the centralizer of the subalgebra $M_n(\mathbb{F})$ in $M_q(\mathbb{F}).$

Consider integers $l $ and $r $  such  that $(l-1)n < i \leq ln,$ $(r-1)n < j \leq rn,$ and let $\overline{i}= i - (l-1)n,$  $\overline{j}= j - (r-1)n.$ Then $$ e_{ij}(\alpha)= \overline{e}_{\overline{i}}\cdot e_{ij}(\alpha)\cdot \overline{e}_{\overline{j}}. $$ If $i-j$ is not divisible by $n$, then $\overline{i}\neq \overline{j}$ and, therefore, $t_{ij}(\alpha)$ is a   transvection of the ring $M_n\big( M_k(\mathbb{F})\big).$ 

If $i-j$ is  divisible by $n$, then there exists an integer  $m,$ $1 \leq m\leq q,$ such that $i-m$  is not divisible by $n$. In this case $t_{ij}(\alpha)=[t_{im}(1),t_{mj}(\alpha)]. $ In any case $t_{ij}(\alpha)\in E_n(A').$ This completes the proof of the lemma. \end{proof}

I.~Z.~Golubchik  and A.~V.~Mikhalev \cite{Gol_Mikh_1}, and E.~I.~Zelmanov \cite{Zel_1} described isomorphisms of elementary linear groups over rings.

\begin{theorem}[\textbf{I.~Z.~Golubchik, A.~V.~Mikhalev, E.~I.~Zelmanov}]\label{GMZ_1} Let $R,$ $S$ be rings with $\frac{1}{6},$  and let $n \geq 4,$ $m \geq 4$ be integers. If $\varphi : E_n(R) \rightarrow E_m(S)$ is an isomorphism of elementary linear groups, then there exist central idempotents $e,$ $f$ in the matrix rings $M_n(R),$ $M_m(S),$ respectively, an isomorphism $\theta_1 : eM_n(R)  \rightarrow fM_m(S)$ and an anti-isomorphism $\theta_2 : (1-e)M_n(R)  \rightarrow (1-f)M_m(S)$ such that $$\varphi(g)=\theta_1(eg)+ \theta_2((1-e)g^{-1}) \quad \text{for an arbitrary element} \quad g\in E_n(R).$$ \end{theorem}

\begin{proof}[Proof of Theorem \emph{\ref{Th_1_per_matrix}}] Let $A$ and $B$ be unital locally matrix algebras, and let $\varphi: [A^{*},A^{*}] \rightarrow [B^{*},B^{*}]$ be an isomorphism. As above, we assume that algebras $A$ and $B$ are infinite-dimensional. Choose positive integers $n \geq 4$ and $m \geq 4$ dividing the Steinitz numbers $\textbf{st}(A)$ and $\textbf{st}(B),$ respectively. Then  algebras $A$ and $B$ can be identified with matrix algebras $M_n(A')$ and $M_m(B'),$ where $A',$ $B'$ are centralizers of  subalgebras $1\in M_n(\mathbb{F})$ and $1\in M_m(\mathbb{F})$ in $A$ and $B,$ respectively.
	
By Lemma \ref{lemma1_per_matrix}, $$[A^{*},A^{*}] = E_n(A'), \quad  [B^{*},B^{*}] = E_m(B'),  $$ and $\varphi : E_n(A') \rightarrow E_m(B') $ is an isomorphism. Since  algebras $A,$ $B$ are simple, their only central idempotents are $0$ and $1.$ By the  Golubchik-Mikhalev-Zelmanov Theorem (see Theorem \ref{GMZ_1}), there exists an isomorphism $\theta_1: A \rightarrow B$ such that $\varphi(g)= \theta_1(g)$ for an arbitrary $g\in [A^{*},A^{*}]  $  or there  exists an anti-isomorphism $\theta_2 :A \rightarrow B$ such that $$\varphi(g) = \theta_2(g^{-1}) \quad \text{for an arbitrary element} \quad  g\in [A^{*},A^{*}].$$ This completes the proof of Theorem \ref{Th_1_per_matrix}. 
\end{proof}

\begin{proof}[Proof of Theorem \emph{\ref{Th_2_per_matrix}}] Let $A,$ $B$ be countable-dimensional unital locally matrix algebras, and let $\varphi: [A^{*},A^{*}] \rightarrow [B^{*},B^{*}]$ be an isomorphism. We have already mentioned that the algebra $A$ can be identified with the algebra $M_s^p (\mathbb{F}),$ where $s=\textbf{st}(M_s^p (\mathbb{F})),$  and that the algebra $A=M_s^p (\mathbb{F})$ is invariant with respect to the transpose $t.$
	
If  $\theta: A \rightarrow B$ is an anti-isomorphism, then the mapping $\theta\,': A \rightarrow B,$ $\theta\,'(a)=\theta(a^t)$ is an isomorphism and $\theta(g^{-1})=\theta' \big((g^{-1})^t\big) $ for an arbitrary element $g\in [A^{*},A^{*}].$ This completes the proof of Theorem \ref{Th_2_per_matrix}. 
 \end{proof}

\begin{lemma}\label{lemma2_per_matrix} Let $A$ and $B$ be countable-dimensional unital locally matrix algebras. If $A$ and $B$ are isomorphic as rings, then they are isomorphic as $\mathbb{F}$-algebras.
 \end{lemma}

\begin{proof} Since the algebra $B$ is countable-dimensional, without loss of generality, we can assume that $B=M_s^p(\mathbb{F}),$ where $s=\textbf{st}(M_s^p(\mathbb{F})).$ An arbitrary automorphism $\tau$ of the field $\mathbb{F}$ extends to an automorphism $\widetilde{\tau}$ of the ring $M_s^p(\mathbb{F}),$ $$\widetilde{\tau}: (a_{ij})_{\mathbb{N}\times \mathbb{N}} \mapsto \big( \, \tau (a_{ij})\, \big)_{\mathbb{N}\times \mathbb{N}}\,.$$  
	
Let  $\theta: A \rightarrow B$ be  an isomorphism of rings. Then $\theta$ maps  the center $\mathbb{F} \cdot 1_A$ of the algebra $A$ to the center  $\mathbb{F} \cdot 1_B$ of the algebra $B.$ Then there exists an automorphism $\tau$ of the field $\mathbb{F}$ such that $\theta(\alpha \cdot 1_A)= \tau(\alpha) \cdot 1_B$ for an arbitrary element $\alpha\in \mathbb{F}.$ The composition $\theta \circ \widetilde{\tau}^{-1}  $ is an isomorphism of $\mathbb{F}$-algebras $A\rightarrow B.$ This completes the proof of the lemma.
 \end{proof}

Before we prove Theorems \ref{Th_4_per_matrix} and \ref{Th_5_per_matrix} we will show how central homotheties arise in locally matrix algebras.

Let  $s$ be a Steinitz number. Suppose that for all integers $n \geq 1$ such that $n|s$ there exist homomorphisms $\tau_n: \mathbb{F}^{*}\rightarrow \mathbb{F}^{*} $ such that 
\begin{enumerate}
	\item[(i)] $\big(\tau_n(\alpha)\big)^n = \alpha$ for all $\alpha \in  \mathbb{F}^{*},$ 
		\item[(ii)] if $m|n,$ $n|s,$ $n=m\cdot k,$ then  $\tau_n(\alpha)=\tau_k\big(\tau_m(\alpha)\big)$ for all $\alpha \in  \mathbb{F}^{*}.$
\end{enumerate} 

For example, if $\mathbb{F}=\mathbb{R}$ is the field of real numbers and $s$ is a Steinitz number that is not divisible by $2,$ then the mapping $a \mapsto a^{\frac{1}{n}},$ $a\in\mathbb{R}, $ is well defined and satisfies $\text{(i)},$ $\text{(ii)}.$ 

For an arbitrary element $a\in A,$ we will define its \textit{relative determinant} $\text{det}_r(a)$ in the following way. There exists a matrix algebra $M_n(\mathbb{F})\subset A$ such that $1, a\in M_n(\mathbb{F}).$ Let $\text{det}_{M_n(\mathbb{F})}(a)$ be the determinant of the matrix $a$ in $M_n(\mathbb{F}).$  Then  $$\text{det}_r(a)= \tau_n \big(\text{det}_{M_n(\mathbb{F})}(a) \big)$$ does not depend on a choice of the subalgebra $M_n(\mathbb{F}).$ Indeed, if $1, a\in M_m(\mathbb{F}) \subset M_n(\mathbb{F}),$ then the subalgebra $M_m(\mathbb{F})$ is embedded in $M_n(\mathbb{F})$ diagonally. Hence, $$\text{det} _{M_n(\mathbb{F})}(a)=\big(\text{det}_{M_m(\mathbb{F})}(a)\big)^k, \quad k=n/m.$$ Denote $\alpha=\text{det}_{M_m(\mathbb{F})}(a).$ Then, by (i) and (ii), $$\tau_n (\alpha^k)=\tau_m \big(\tau_k (\alpha^k) \big)= \tau_m (\alpha).$$

It is easy to see that $\text{det}_r(ab)=\text{det}_r(a) \text{det}_r(b)$ for arbitrary elements $ a,b \in A,$ and $ \text{det}_r(1)=1.$ An element $a\in A$ is invertible if and only if $\text{det}_r(a)\neq 0.$ The mapping $$A^{*} \rightarrow \mathbb{F}^{*} , \quad a \mapsto \text{det}_r(a),$$ is central homothety.

Let $R$ be a ring with $1,$ and let $n \geq 2$ be an integer. Along with the elementary linear group $E_n(R)$ consider the group $GE_n(R)$ that is generated by $E_n(R)$ and all invertible diagonal matrices over $R.$

The following result is due to I.~Z.~Golubchik  and A.~V.~Mikhalev~\cite{Gol_Mikh_1}, and 	E.~I.~Zelmanov~\cite{Zel_1}.
\begin{theorem}[\textbf{I.~Z.~Golubchik, A.~V.~Mikhalev, E.~I.~Zelmanov}]\label{GMZ_2} Let $R,$ $S$ be rings with $\frac{1}{6},$   let $m\geq 4,$ $n\geq 4$ be integers, and let $\varphi :GE_n(R)\rightarrow GE_m(S)  $ be an isomorphism. Then there exist central idempotents $e$ and $f$ of  matrix rings $M_n(R)$ and $M_m(S),$ respectively, an isomorphism $\theta_1 : eM_n(R)\rightarrow f M_m(S), $ an anti-isomorphism $\theta_2 : (1-e) M_n(R)\rightarrow (1-f) M_m(S) $ and a homomorphism $\chi: GE_n(R) \rightarrow Z(GE_m(S))$ to the center of the group $GE_m(S)$ such that $$\varphi(g)=\chi(g) \big(\theta_1 (eg)+ \theta_2((1-e)g^{-1})\big) \quad \text{for an arbitrary element} \quad g\in GE_n(R). $$ 
\end{theorem}

Let $A$ be a unital locally matrix algebra,  let $1\in M_n(\mathbb{F}) \subset A,$ and let $n\geq 4.$ Consider the centralizer $A'$ of the subalgebra $M_n(\mathbb{F})$ in $A.$ As above, we identify the algebra $A$ with the matrix algebra $M_n(A').$

\begin{lemma}\label{lemma3_per_matrix} $A^{*}=GE_n (A').$
\end{lemma}
\begin{proof} Let $g\in A^{*}.$ There exists a matrix subalgebra $M_q(\mathbb{F})\subset A$ such that $M_n(\mathbb{F}) \subset M_q(\mathbb{F}),$ $g\in M_q(\mathbb{F}).$ Then $g\in \big(M_q(\mathbb{F})\big)^{*}=GL_q(\mathbb{F}).$ It is well known that the group $GL_q(\mathbb{F})$ is generated by transvections and diagonal matrices $$ d_{11}(\alpha)= \text{diag}\, (\underbrace{\alpha,1,1, \ldots, 1}_q\,), \quad 0 \neq \alpha \in\mathbb{F}.$$
As in the proof of Lemma \ref{lemma1_per_matrix}, consider  matrix units $e_{ii}(1),$ $1 \leq i \leq n,$  of the algebra $M_n(\mathbb{F}).$ They are embedded in  $M_q(\mathbb{F})$ as idempotents 
 $$ \overline{e}_{i}= \text{diag}\, (\underbrace{e_{ii}(1), \ldots, e_{ii}(1)}_{q/n}\,).$$ Clearly, $$ d_{11}(\alpha)-I_q = \overline{e}_{1} \, \big(d_{11}(\alpha)-I_q \big)\, \overline{e}_{1}.$$ Hence, $d_{11}(\alpha)$ is an invertible diagonal matrix of $M_n\big( A' \cap M_q(\mathbb{F})\big),$ $d_{11}(\alpha)  \in GE_n(A').$ We showed that $GL_q(\mathbb{F})\subset GE_n(A')$ and, therefore, $g\in GE_n(A').$  This completes the proof of the lemma.
\end{proof}

Now, Theorems \ref{Th_4_per_matrix} and \ref{Th_5_per_matrix} immediately follow from simplicity of the algebras $A$ and $B,$ Lemma \ref{lemma3_per_matrix} and the results of I.~Z.~Golubchik  and A.~V.~Mikhalev~\cite{Gol_Mikh_1}, and 	E.~I.~Zelmanov~\cite{Zel_1} (Theorem~\ref{GMZ_2}).

\section{Isomorphisms of groups of  periodic  infinite matrices}

Now, let us describe the group of automorphisms of $SL_s^p(\mathbb{F}).$  

\begin{proof}[Proof of Theorem \emph{\ref{Th_6_per_matrix}}] Let $\text{Aut}_{\text{ring}}\big(M_s^p(\mathbb{F})\big)$ be the group of ring automorphisms of $M_s^p(\mathbb{F}).$ 
	
We claim that $$\text{Aut}\big(SL_s^p(\mathbb{F})\big)= H\cdot \text{Aut}_{\text{ring}}\big(M_s^p(\mathbb{F})\big).$$ Indeed, by Theorem \ref{Th_2_per_matrix}, an arbitrary automorphism  $\varphi\in \text{Aut}\big(SL_s^p(\mathbb{F})\big)$ either extends to an automorphism of the ring $M_s^p(\mathbb{F}),$ in which case $\varphi\in \text{Aut}_{\text{ring}}\big(M_s^p(\mathbb{F})\big),$ or there exists an automorphism $\theta \in\text{Aut}_{\text{ring}}\big(M_s^p(\mathbb{F})\big) $ such that  $$\varphi(g)=\theta\big( (g^{-1})^t\big) \quad \text{for all} \quad g\in SL_s^p(\mathbb{F}),$$ in which case $\varphi =\psi \circ\theta \in H \cdot\text{Aut}_{\text{ring}}\big(M_s^p(\mathbb{F})\big) .$ 

Let us show that $$\text{Aut}_{\text{ring}}\big(M_s^p(\mathbb{F})\big)=\text{Aut}_{\mathbb{F}}\big(M_s^p(\mathbb{F})\big)\cdot \text{Aut}(\mathbb{F}).$$ In the proof of Lemma \ref{lemma2_per_matrix} we showed that an arbitrary automorphism $\tau$ of the field $\mathbb{F}$ gives rise to the automorphism $\widetilde{\tau}$ of the ring $M_s^p(\mathbb{F}).$ The mapping $\tau \mapsto \widetilde{\tau}$ is an embedding of the group $\text{Aut}({\mathbb{F}})$ into the group $\text{Aut}_{\text{ring}}\big(M_s^p(\mathbb{F})\big).$

If $\varphi$ is an automorphism of the ring $M_s^p(\mathbb{F}),$ then its restriction $\varphi\,\big|\,_{\mathbb{F}\cdot 1}$ to the center $\mathbb{F} \cdot 1$ of the ring   $M_s^p(\mathbb{F})$ is an automorphism of the field $\mathbb{F},$ and it gives rise to the automorphism $\widetilde{\varphi\,\big|\,_{\mathbb{F}\cdot 1}}.$  Clearly, $$\varphi \cdot \big(\widetilde{\varphi\,\big|\,_{\mathbb{F}\cdot 1}} \big)^{-1} \in \text{Aut}_{\mathbb{F}}\big(M_s^p(\mathbb{F})\big),$$ which implies the assertion of  Theorem \ref{Th_6_per_matrix}.
\end{proof}

\end{document}